\documentclass[12pt,a4paper,reqnol]{amsart}
\usepackage{amsmath,amsthm,amssymb}
\usepackage{a4wide}
\usepackage{hyperref}
\usepackage{enumerate}
\usepackage[all]{xy}

\usepackage{color}
\usepackage{ifpdf}
\ifpdf %if using pdfLaTeX in PDF mode
  \usepackage[pdftex]{graphicx}
  \DeclareGraphicsExtensions{.pdf,.png,.mps}
  \usepackage{pgf}
  \usepackage{tikz}
\else %if using LaTeX or pdfLaTeX in DVI mode
  \usepackage{graphicx}
  \DeclareGraphicsExtensions{.eps,.bmp}
  \usepackage{pgf}
  \usepackage{tikz}
  \usepackage{pstricks}%variant: \usepackage{pst-all}
\fi
\usepackage{epic}
\usepackage{bez123}
\newtheorem{thm}{Theorem}%[section]
\newtheorem{cor}[thm]{Corollary}

\newtheorem{lem}[thm]{Lemma}

\theoremstyle{definition}

\theoremstyle{remark}
\newcommand{\abs}[1]{\left\vert#1\right\vert}

\usepackage{amsmath}
\DeclareMathOperator{\inv}{inv}
\DeclareMathOperator{\tinv}{\mathbf{inv}}
\DeclareMathOperator{\jump}{jump}
\DeclareMathOperator{\tjump}{\mathbf{jump}}
\DeclareMathOperator{\lead}{lead}
\DeclareMathOperator{\lucky}{lucky}

\DeclareMathOperator{\tree}{tree}
\DeclareMathOperator{\critic}{critical}
\usepackage{amssymb}

\newcommand{\cirnum}[1]{\textcircled{\tiny $#1$}}
\newcommand{\boxnum}[1]{\framebox[1em][c]{$#1$}}
\newcommand{\plainnum}[1]{\makebox[1em][c]{$#1$}}

\begin{document}
\title{A New Bijection Between Forests and Parking Functions}%
\author{Heesung Shin}
\address{Universit\'e de Lyon; Universit\'e Lyon 1; Institut Camille Jordan, CNRS
UMR 5208; 43 boulevard du 11 novembre 1918, F-69622 Villeurbanne Cedex, France}
\email{hshin@math.univ-lyon1.fr}

%\thanks{}%
\subjclass[2000]{05A15}%
%\keywords{}%

\date{\today}
%\dedicatory{}%
%\commby{}%
% ----------------------------------------------------------------
\begin{abstract}
In 1980, G. Kreweras \cite{MR603398} gave a recursive bijection between forests and parking functions.
In this paper we construct a nonrecursive bijection from forests onto parking functions, which answers a question raised by R. Stanley \cite[Exercise 6.4]{MR2383131}. As a by-product, we obtain a bijective proof of Gessel and Seo's formula for lucky statistic on parking functions \cite{MR2224940}.

%This paper investigates the statistics of \emph{lucky} of parking functions. The generating function for lucky is proved by Gessel and Seo but their proof is not combinatorial. In this paper, we construct new bijection from forests to parking functions and give a bijective proof of the formula
%$$\sum_{F} q^{\inv(F)} u^{\lead(F)} c^{\tree(F)}= \sum_{P} q^{\jump(P)} u^{\lucky(P)} c^{\critic(P)}$$
%where the sum is over all forest $F$ on $n$ vertices and all parking function $P$ of length $n$.
\end{abstract}
\maketitle
% ----------------------------------------------------------------

\section{Introduction}
%It is well-known that the numbers of forests on $n$ vertices and parking functions with length $n$ are same.
It is well-known (see e.g. \cite{MR0335294}) that there are several bijections between forests on $n$ vertices and parking functions with length $n$.
In 1980, G. Kreweras \cite{MR603398} presented his work that connected recursively inversion enumerators for trees with parking functions. After that this recursive bijection was also rewritten in R. Stanley's lecture note \cite{MR2383131} in 2004. In this book, he wrote that we need a \emph{nonrecursive} bijection $\varphi$ between the set $F_n$ of all rooted forests on $n$ vertices and the set $PF_n$ of all parking functions of length $n$ satisfying $$\inv(F) = {n+1 \choose 2} - a_1 - \cdots - a_n$$ where $\varphi(F) = (a_1,\ldots, a_n).$ (See \cite[Exercise 6.4]{MR2383131}) He mentioned that a ``nonrecursive'' bijection would be greatly preferred.

%For solving this problem, we have to find not only a ``nonrecursive'' bijection but also a proper statistic for parking functions $(a_1,\ldots, a_n)$, which is equal to $${n+1 \choose 2} - a_1 - \cdots - a_n.$$

Gessel and Seo \cite{MR2224940} studied the statistic \emph{lucky} of parking functions. The generating function for \emph{lucky} is
\begin{equation} \label{GS}
\sum_{P \in PF_n} u^{\lucky P} = u \prod_{i=1}^{n-1} (i+(n-i+1) u),
\end{equation}
where the sum is over all parking function $P$ of length $n$. This formula is proved by them, but that is not bijective. On the other side, Seo and Shin \cite{MR2353128} introduced the statistic \emph{leader} of forests, and whose generating function is
\begin{equation} \label{SS}
\sum_{F \in F_n} u^{\lead F} = u \prod_{i=1}^{n-1} (i+(n-i+1) u),
\end{equation}
where the sum is over all forest $F$ on $n$ vertices, which is proved bijectively using reverse Pr\"ufer code. Since the right-hand sides of two equations \eqref{GS} and \eqref{SS} are same, we have to find a bijection between forests and parking functions which yields
\begin{equation}
\sum_{F \in F_n} u^{\lead F} = \sum_{P \in PF_n} u^{\lucky P}.
\end{equation}
%the equation \eqref{GS} would have the bijective proof.
%In order to solve above, we will find a bijection between them with preserving each statistic of them.

In this paper, we construct a nonrecursive bijection $\varphi:F_n \to PF_n$ between forests and parking functions satisfying
\begin{eqnarray*}
\inv(F) &=& \textstyle {n+1 \choose 2} - \abs{P}\\
\lead(F) &=& \lucky(P)
\end{eqnarray*}
where $P = \varphi(F)$ and  $\abs{P}$ is the sum of sequences.

%By the result, we also have
%\begin{equation}
%\sum_{F \in F_n} q^{\inv(F)} u^{\lead(F)} = \sum_{P \in PF_n} q^{\jump(P)} u^{\lucky(P)}.
%\end{equation}

Moreover, reviewing the bijection $\varphi$, it has been observed that parking functions have a statistic corresponding to the statistic {\em tree}, the number of trees, in forests. When this statistic in parking functions is called {\em critical}, the bijection $\varphi$ preserves the statistics $\inv$, $\lead$, and $\tree$ for forests to $\jump$, $\lucky$, and $\critic$ for parking functions.

%In other words, we get a more general formula
%\begin{equation}
%\sum_{F \in F_n} q^{\inv(F)} u^{\lead(F)} c^{\tree(F)}= \sum_{P \in PF_n} q^{\jump(P)} u^{\lucky(P)} c^{\critic(P)}.
%\end{equation}

\section{Preliminary}\label{preliminary}
%If vertices on a tree are labeled, we call it a \emph{labeled} tree and if only one vertex among its vertices is chosen as root, we call it a \emph{rooted} tree.
A graph on labeled vertices is called {\em labeled} and if a graph have one distinguished vertex, the vertex is called {\em root} and the graph is called {\em rooted}.
A \emph{tree} is a simple connected rooted labeled graph without cycles. A {\em forest} is a graph in which any two vertices are connected by at most one path and each connected component is a tree.
%Let $F_n$ be the set of forests on $n$ vertices.

%A {\em leader} of forest is a vertex which has no smaller descendants.
%When we draw a rooted graph, we usually draw the root at the top. We call it {\em "hang up"}.

A vertex $j$ is called a \emph{descendant} of a vertex $i$ if a vertex $i$ lies on the unique path from the root to a vertex $j$. This is equivalent to the statement that a vertex $i$ is a \emph{ascendant} of a vertex $j$. An \emph{inversion} in a rooted graph is an ordered pair $(i,j)$ such that $i>j$ and $j$ is a descendant of $i$. Let $\inv(G:v)$ denote the number of ordered pairs $(v,x)$ where $v>x$ and $x$ is a descendant of $v$ in a rooted graph $G$ and $\inv(G)$ the number of all inversions in a rooted graph $G$. By definition, $\inv(G) = \sum_{v \in G} \inv(G:v)$.

An vertex $v$ is called a \emph{leader} in a rooted graph $G$ if $\inv(G:v)=0$, that is, the vertex $v$ is the smallest among its all descendants. By definition, every leaf is a leader. $\lead(G)$ denotes the number of all leaders in a rooted graph $G$.

%Look at this rooted tree $T$ with root $4$, for example.
%\input{tree1.TpX}
%Then we know that all inversions of $T$ are $(4,1)$, $(4,2)$, and $(4,3)$ and every vertex except 4 is a leader of tree $T$. Thus $\inv(T)=3$ and $\lead(T)=5$.

%\subsection*{RP-order of vertices in a tree}
%We will consider the order of vertices in a tree. In making Pr\"ufer Code from a tree, we remove \emph{the smallest leaves}. In this time, we naturally get the order of deleting vertices. The first vertex is the smallest leaf and the last vertex is the root in a tree. We will call it \emph{Pr\"ufer order}. \emph{RP-order} means `reverse Pr\"ufer order', that is, really the reverse order of Pr\"ufer order. Trivially, the first in RP-order is the root in a tree. We can see that RP-order can be obtained by traveling from the root toward to the largest unvisited vertex in a tree. Actually, if the vertex just visited is smallest of unvisited vertices, it becomes the \emph{leader}.
%
%\begin{figure}[t]
%\input{tree2.TpX}
%\caption{Pr\"ufer order is $(1, 3, 4, 2, 5, 6)$ and RP-order is $(6, 5, 2, 4, 3, 1)$.}
%\end{figure}

\subsection*{3 rules drawing a forests}
We want that the shape of a tree is unique by drawing in only one way. When we draw a forest, we keep the following rules.
\begin{itemize}
\item Draw the roots at the top and all trees grow downward.
\item Put the trees from left to right according to maximum label in each tree.
\item Similarly, when vertices are drawn, put siblings from left to right in the order of maximum labels in their descendants.
\end{itemize}
It seems that the shape of a forest is figured as a rooted ordered forest after drawing.
%A {\em parking function} is a finite integer sequence $(a_1, a_2, \ldots, a_n)$ if there exists $(b_1, b_2, \ldots, b_n)$ which is a rearrangement of $(a_1, a_2, \ldots, a_n)$ such that $b_1 \le b_2 \le \cdots \le b_n$ and $1 \le b_i \le i$ for all $i$.
%After {\em Parking Algorithm}, a car is {\em lucky} if it succeeds in parking in its favorite space.

\subsection*{Parking algorithm}
Given a sequence $(p_1,p_2,\ldots,p_n)$ of length $n$, where $p_i$ means the favorite parking space of the $i$-th car, we can park $n$ cars into parking spaces as follows:
\begin{enumerate}
%\item This parking space is an one-way road from left to right and does not allow that a car is going back.
%\item Each car has its favorite parking space.
%\item Two or more cars cannot be parked at one parking space.
\item Cars can be parked one by one from the first car to the last car into infinitely many parking spaces whose entrance is at the left.
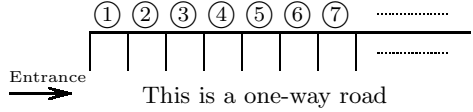
\begin{center}
\input{park_space0.TpX}
\end{center}
\item When $i$-th car is parked, a car has to reach at its favorite parking space $p_i$. And then, attempt to be parked there. If the space is empty, the car is parked. Otherwise, attempt again at the next parking space without going back. Repeat this process until success to park.
\item Let $q_i$ be the actual parking space with $i$-th car.
\end{enumerate}
This method is called a \emph{Parking Algorithm} and the notation \emph{$PA$} is defined by $$PA(p_1,\ldots,p_n)=(q_1,\ldots,q_n).$$
For example, given a sequence $(4,3,3,1,5)$, five cars can be parked by the Parking Algorithm as following.
\input{park_example.TpX}
We get a sequence with length $5$, $PA(4,3,3,1,5)=(4,3,5,1,6)$.

If $PA(p_1,\ldots,p_n)=(q_1,\ldots,q_n)$ and the actual parking spaces $q_i$ is less than or equal to $n$ for all $i$, the sequence $(p_1,\ldots,p_n)$ is called a \emph{parking function}.
%Let $PF_n$ denote the set of parking functions with length $n$.

%By definition, we know clearly that every element of parking function is less than or equal to $n$, that is, $p_i\le n$.
%each of $n$ cars can be parked at each of first $n$ parking spaces,
%Note that the number of $PF_n$, $\abs{PF_n} = (n+1)^{n-1}$, is well-known.

%\begin{thm}[Other definitions of parking functions]
%Given a sequence $P=(p_1,p_2,\ldots,p_n)$, the following are equivalent.
%\begin{enumerate}
%\item $P$ is a parking function. \item \emph{$PA(P)$} can be considered as a \emph{permutation} of length $n$. \item For all $k\le n$, $\abs{\set{i\,|\,p_i\le k}} \ge k$. \item If $p'_1\le p'_2\le \cdots \le p'_n$ is the rearrangement of $P$, then $p'_k \le k$ for all $k$.
%\end{enumerate}
%\end{thm}

A \emph{jump} in a parking function is defined by the attempt to park the next space because of a non-empty parking space. Let $\jump(P:c)$ be the number of the jumps in order to park the car $c$, that is, the difference between the favorite parking space $p_c$ and the actual parking space $q_c$. So we make the formula $$\jump(p_1,\ldots,p_n:c) = q_c - p_c$$ where $PA(p_1,\ldots,p_n)=(q_1,\ldots,q_n)$. And $\jump(P)$ denotes the number of the jumps to park all cars. By definition, $\jump(P) = \sum_{c} \jump(P:c)$. Therefore, we have
$$\jump(P) = \sum_{c} \jump(P:c) = \sum_{c} q_c - p_c = {n+1 \choose 2} - \abs{P}$$
where $\abs{P} = \sum p_i$.

A \emph{lucky} in a parking function $P$ is the car $c$ where $\jump(P:c)=0$, that is, the car $c$ is parked at its favorite parking space. $\lucky(P)$ denotes the number of all lucky cars in $P$.
For example, for a given parking function $P=(2,4,2,1,3)$, we get the sequence $PA(P)=(2,4,3,1,5)$ by the parking algorithm.
\begin{center}
\begin{tabular}{c|l}
$q_c$     & \cirnum{1} \cirnum{2} \cirnum{3} \cirnum{4} \cirnum{5}\\
$c$       & \plainnum{4} \plainnum{1} \plainnum{3} \plainnum{2} \plainnum{5}  \\ \hline
$p_c$     & \cirnum{1} \cirnum{2} \cirnum{2} \cirnum{4} \cirnum{3}\\
$q_c-p_c$ & \boxnum{0} \boxnum{0} \boxnum{1} \boxnum{0} \boxnum{2}\\
\end{tabular}
\end{center}
With seeing above table, we can calculate the jump and the lucky. Actually, we have $\jump(P)=3$ since $\jump(P:3)=1$ and $\jump(P:5)=2$. Also, we have $\lucky(P)=3$ since lucky cars are 1, 2, and 4.

\section{The Map $\varphi : F_n \to PF_n$}
%\subsection{Description}

First of all, the map $\varphi$ is defined according to the diagram in Figure~\ref{diagram1}.
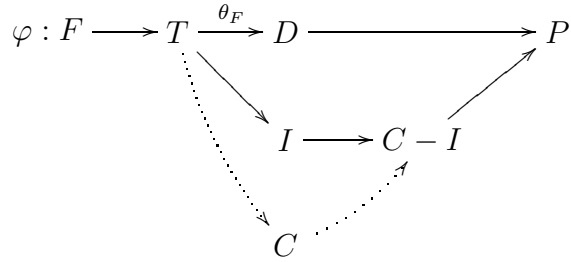
\begin{figure}[t]
\centering \mbox{\xymatrix{
\varphi : F \ar[r] & T \ar[r]^{\theta_F} \ar[dr] \ar@/_/@{.>}[ddr] & D \ar[rr]    &      & P\\
                   &                     & I \ar[r]          & C-I  \ar[ur]\\
                   &                     & C \ar@{.>}@/_/[ur] \\
}} \caption{Diagram of $\varphi$} \label{diagram1}
\end{figure}
Considering one forest $F\in F_{14}$ as an example, we are going to describe how to define the map~$\varphi$ as follows:
\begin{enumerate}
%\item Draw the forest $F \in F_n$ according to the method that we decide in Section~\ref{preliminary}.
\item Draw the forest $F \in F_n$ by 3 rules in Section~\ref{preliminary}.
    \input{map_F.TpX}
%\item Add the vertex $n+1$ at the top and change the forest $F$ to the tree $T$. \input{map_T.TpX} \item Rearrange the label on vertices by the following pseudo-code:
\item Change the forest $F$ to the tree $T$, adding the vertex $n+1$ at the top and connecting new vertex to each root of trees in $F$.
    \input{map_T.TpX}
\item Rearrange the label on vertices by the following pseudo-code:
    \begin{quote}
    for all $v\in V$ do
    \begin{enumerate}[i)]
    \item find the maximum label $m$ on descendants of $v$.
    \item label $m$ on $v$.
    \item rearrange the other labels in descendants of $v$ by order-preserving.
    \end{enumerate}
    end do
    \end{quote}
    %For example, considering $v$ labeled by $8$, the vertex $v$ is labeled by $13$ instead of $8$ after rearranging labels on descendants of $v$.
    For example, after rearranging labels on descendants of $v$ labeled by $8$ in the tree $T$, we label $13$, the maximum of descendants on $v$, on the vertex $v$.
    \input{map_process.TpX}
    This is well-defined, that is independent to an order of choosing vertices $v \in V$.

\item The decreasing tree $D$ is made after acting above process on all vertices. The map $\theta_F$ is induced by the correspondence of labels in a tree after relabeling. For example, $\theta_F(10)=9$.
    \input{map_D.TpX}

\item Because we cannot remake the original tree $T$ from the tree $D$ alone, we need another tree induced
    from the unused information of $T$, that is, $\inv(T:v)$. So we make a new tree $I$ such that each
    vertice $v$ is labeled with $\inv(T:v)$. In order to distinguish it from other labels, we use the box.
    \input{map_I.TpX} Note that we can produce the original tree $T$ from only two trees $D$ and $I$.

\item Additionally, label the vertices indexed by post-order which is indicated by circle. The tree $C$ is
    only dependent to the {\em underlying graph} of $T$, that is, its tree structure. This is the reason why
    we define the method to redraw the tree in the inverse map $\varphi^{-1}$. \input{map_C.TpX}

\item Finally, we make the tree $D\times (C-I)$, in which the plain labels are induced by $D$ and the circled
    labels are induced by $C$ subtracted by $I$.
    \input{map_DCI.TpX}

\item In the sequel, we delete the tree-structure from $D\times (C-I)$ and write the circle number by the order of plain number. Since the last is always \cirnum{1}, we delete it and the rest becomes a {\em parking function}.
    Although all labels of $I$ are zeros in the worst case, the set of circle labels of $C-I$ becomes $[n]$. Since every permutation is a parking function, it becomes a parking function in the worst case.
    For continuing example, we get the sequences
    \begin{center}
    \begin{tabular}{l|l}
    \plainnum{1} \plainnum{2} \plainnum{3} \plainnum{4} \plainnum{5} \plainnum{6} \plainnum{7} \plainnum{8} \plainnum{9} \plainnum{10}
    \plainnum{11} \plainnum{12} \plainnum{13} \plainnum{14} & \plainnum{15}\\
    \cirnum{10} \cirnum{2} \cirnum{6} \cirnum{5} \cirnum{7} \cirnum{1} \cirnum{13} \cirnum{10} \cirnum{4} \cirnum{1}
    \cirnum{14} \cirnum{9} \cirnum{11} \cirnum{5} & \cirnum{1}\\
    \end{tabular}
    \end{center}
    Below the plain label \plainnum{15}, there is always circled label \cirnum{1}. So we can omit it, and then the second row (circled label) becomes a parking function $P$ of length $14$.
    $$P=\cirnum{10} \cirnum{2} \cirnum{6} \cirnum{5} \cirnum{7}
    \cirnum{1} \cirnum{13} \cirnum{10} \cirnum{4} \cirnum{1} \cirnum{14} \cirnum{9} \cirnum{11} \cirnum{5}$$ \end{enumerate}
%\subsubsection{Summary of the map $\varphi$}
%\input{total.TpX}
%\begin{center}
%\begin{tabular}{rl|l}
%&\plainnum{1} \plainnum{2} \plainnum{3} \plainnum{4} \plainnum{5}
%\plainnum{6} \plainnum{7} \plainnum{8} \plainnum{9} \plainnum{10}
%\plainnum{11} \plainnum{12} \plainnum{13} \plainnum{14} & \plainnum{15}\\
%P=&\cirnum{10} \cirnum{2} \cirnum{6} \cirnum{5} \cirnum{7}
%\cirnum{1} \cirnum{13} \cirnum{10} \cirnum{4} \cirnum{1}
%\cirnum{14} \cirnum{9} \cirnum{11} \cirnum{5} & \cirnum{1}\\
%\end{tabular}
%\end{center}

\section{The Inverse Map $\varphi^{-1} : PF_n \to F_n$}
In this section, we construct the inverse map $\varphi^{-1}$ from parking functions to forests as Figure~\ref{diagram2}. We start from the previous example $P \in PF_{14}$, $$P=\cirnum{10} \cirnum{2} \cirnum{6} \cirnum{5} \cirnum{7} \cirnum{1} \cirnum{13} \cirnum{10} \cirnum{4} \cirnum{1} \cirnum{14} \cirnum{9} \cirnum{11} \cirnum{5}.$$
\begin{figure}[t]
\centering \mbox{\xymatrix{
\varphi^{-1} :  P \ar[r]^{PA} & PA(P) \ar[r] \ar[dr] \ar@/_/@{.>}[ddr] & D \ar[r] & T \ar[r] & F \\
                              &                                        & I \ar[ur] \\
                              &                                        & C \\
}} \caption{Diagram of $\varphi^{-1}$} \label{diagram2}
\end{figure}
After adding the \cirnum{1} at the end of $P$, 15 cars are parked by the parking algorithm as follows:
\begin{center}
\begin{tabular}{ll|l}
Parking Space& \cirnum{1} \cirnum{2} \cirnum{3} \cirnum{4} \cirnum{5} \cirnum{6} \cirnum{7} \cirnum{8} \cirnum{9} \cirnum{10}
\cirnum{11} \cirnum{12} \cirnum{13} \cirnum{14} & \cirnum{15}\\
Cars' Number, $c$& \plainnum{6} \plainnum{2} \plainnum{10} \plainnum{9} \plainnum{4} \plainnum{3} \plainnum{5} \plainnum{14} \plainnum{12} \plainnum{1}
\plainnum{8} \plainnum{13} \plainnum{7} \plainnum{11} & \plainnum{15}\\
\end{tabular}
\end{center}
At this time, we record the jump for every car in third row. And then, we draw an edge between the car $c$ and the closest car on its right which is larger than $c$.
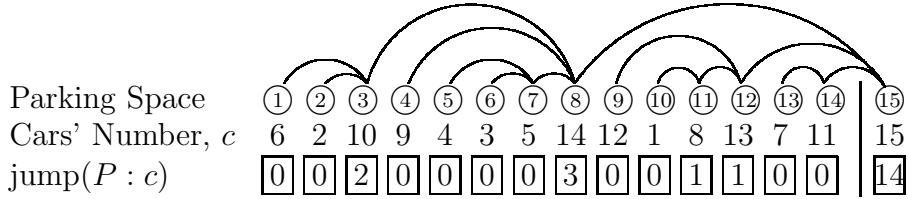
\begin{center}
\begin{tabular}{ll|l}
&\multicolumn{2}{c}{\input{edge.TpX}}\vspace{-1em}\\
Parking Space& \cirnum{1} \cirnum{2} \cirnum{3} \cirnum{4} \cirnum{5} \cirnum{6} \cirnum{7} \cirnum{8} \cirnum{9} \cirnum{10}
\cirnum{11} \cirnum{12} \cirnum{13} \cirnum{14} & \cirnum{15}\\
Cars' Number, $c$& \plainnum{6} \plainnum{2} \plainnum{10} \plainnum{9} \plainnum{4} \plainnum{3} \plainnum{5} \plainnum{14} \plainnum{12} \plainnum{1}
\plainnum{8} \plainnum{13} \plainnum{7} \plainnum{11} & \plainnum{15}\\
$\jump(P:c)$ & \boxnum{0} \boxnum{0} \boxnum{2} \boxnum{0} \boxnum{0} \boxnum{0} \boxnum{0} \boxnum{3} \boxnum{0} \boxnum{0}
\boxnum{1} \boxnum{1} \boxnum{0} \boxnum{0} & \boxnum{14}\\
\end{tabular}
\end{center}
We get the tree-structure on the cars as vertices. If we consider $15$ as a root, we can rebuild three trees $C$, $D$, and $I$. The parking space in the first row becomes the tree $C$, the cars' number in the second row becomes the tree $D$, and the jump in the third row becomes the tree $I$.

Because the label $I(v)$ of vertex $v$ in the tree $I$ stands for $\inv(T:v)$, we can make the tree $T$ from the decreasing tree $D$ as follows:
\begin{quote}
for all vertex $v$ in the tree $D$ do
\begin{enumerate}[i)]
\item find $(I(v)+1)$-th smallest label $m$ on descendants of $v$. \item label $m$ on $v$. \item rearrange the other labels in descendants of $v$ by order-preserving.
\end{enumerate}
end do
\end{quote}
The pseudo-code above is an inverse map of $\theta_F$. After we make the tree $T$, we can get the forest $F$ from $T$ by deleting the maximum vertex of $T$.

\begin{thm}
The above algorithm from a parking function to a forest is the inverse map of $\varphi$.
\end{thm}
\begin{proof}
It is enough to show that the tree-structure deleted in the map $\varphi$ and the tree-structure made in the inverse map $\varphi^{-1}$ are the same. If all labels of the tree $I$ are zeros, the circled labels in the tree $D\times (C-I)$ are distinct. If so, a parking function $P$ is a permutation, that is, all cars are lucky. If $P$ is a permutation, $PA(P)= P$. A parent of a car $c$ is larger than $c$ since the tree $D$ is decreasing and it is on the right of $c$ after parking algorithm because of a post-order. In this time, we can make the tree $D$ from a permutation $P^{-1}$.

Since the tree $D$ is decreasing, all cars corresponding to descendants of $v$ already parked when a car corresponding to $v$ is parking. Using labels of the tree $C-I$ instead of the tree $C$, a favorite parking space of car $c$ corresponding to a vertex $v$ decreases by $\inv(T:v)$ but parking space at which car $c$ parks actually is not changed. Hence we can make the tree $D$ from $PA(P)$.
\end{proof}

\section{Statistics}

%\subsubsection{Relations of statistics}

After we observe the map $\varphi$, we can get Lemma~\ref{lemma}
\begin{lem}
\label{lemma}
The map $\varphi$ has two following properties.
\begin{itemize}
\item $\inv(F:v) = \jump(\varphi(F):\theta_F(v))$ \item If $v$ is a root of a tree in $F$, then $\theta_F(v)$ is a right-to-left maximum in $PA(\varphi(F))^{-1}$.
\end{itemize}
\end{lem}
\begin{proof}
If we use the labels of the tree $C$ instead of the tree $C-I$, all cars are lucky. Using labels of the tree $C-I$ instead of the tree $C$, $\jump(P:c)$ increases by $\inv(T:v)$. Thus $\inv(F:v) = \jump(\varphi(F):\theta_F(v))$.

If a vertex $v$ is a root of a tree in $F$, a parent of $v$ is the root of $T$. So there is no car larger than the car $\theta_F(v)$ on its right. Hence the car $\theta_F(v)$ is a right-to-left maximum in $PA(\varphi(F))^{-1}$.
\end{proof}

Let $\tinv(F)$ be a {\em type of inversions of $F$} and $\tjump(P)$ be a {\em type of jumps of $P$} defined by \begin{eqnarray*}
\tinv(F) &=& (\lead_0(F),\ldots, \lead_n(F))\\
\tjump(P) &=& (\lucky_0(P),\ldots, \lucky_n(P)
\end{eqnarray*}
where $\lead_i(F)$ is the number of vertices $v$ such that $\inv(F:v)=i$ and $\lucky_i(F)$ is the number of cars $c$ such that $\jump(P:c)=i$.

The car $c$ is called {\em critical} if there is no empty parking space on the right of the car $c$ after it is parked. Let $\tree(F)$ be the number of trees in a forest $F$ and $\critic(P)$ be the number of critical cars in a parking function $P$. Note that any critical car becomes a right-to-left maximum in $PA(P)^{-1}$ and its converse is also true.

%The main result is induced by above facts, automatically.

\begin{thm}[Main Theorem]
\label{map}
There is a nonrecursive bijection $\varphi:F_n \to PF_n$ between forests and parking functions satisfying
$$(\tinv,\tree)(F) = (\tjump,\critic)(\varphi(F))$$
%\begin{eqnarray*}
%\tinv(F) &=& \tjump(\varphi(F))\\
%%\lead(F) &=& \lucky(\varphi(F))\\
%\tree(F) &=& \critic(\varphi(F))
%\end{eqnarray*}
\end{thm}
\begin{proof}
By Lemma~\ref{lemma}, there is the correspondence $\theta_F$ between all vertices $v$ in the forest $F$ and all cars $c$ in the parking function $\varphi(F)$ such that
$\inv(F:v) = \jump(\varphi(F):\theta_F(v))$.
So we have $\lead_i(F)=\lucky_i(\varphi(F))$ for all $i=0,\ldots,n$ and $$\tinv(F)=\tjump(\varphi(F)).$$
%\begin{eqnarray*}
%\inv(F:v) &=& \jump(P:c)\\
%&=& \jump(\varphi(F):\theta_F(v)).
%\end{eqnarray*}

By the map $\theta_F$, each root of trees in $F$ corresponds to each of right-to-left maximums in $PA(\varphi(F))^{-1}$. Hence we have $$\tree(F) = \critic(\varphi(F)).$$

%So this is proved with Theorem~\ref{main}.
\end{proof}

%\section{Further Results}
Let $I_n$ and $J_n$ be homogeneous polynomials of degree $n$,
\begin{eqnarray*}
I_n(\mathbf{q};c) &=& \sum_{F \in F_n} \mathbf{q}^{\tinv(F)} c^{\tree(F)}\\
J_n(\mathbf{q};c) &=& \sum_{P \in PF_n} \mathbf{q}^{\tjump(P)} c^{\critic(P)}
\end{eqnarray*}
where $\mathbf{q}^{\mathbf{v}} = q_0^{v_0} q_1^{v_1}\cdots q_n^{v_n}$.

\begin{thm}
\label{main}
For a nonnegative integer $n$, we have
$$I_n(\mathbf{q};c) = J_n(\mathbf{q};c)$$
\end{thm}
\begin{proof}
By Theorem~\ref{map}, there exists the bijection $\varphi:F_n \to PF_n$ such that $$\mathbf{q}^{\tinv(F)} c^{\tree(F)} = \mathbf{q}^{\tjump(\varphi(F))} c^{\critic(\varphi(F))}.$$ So we have $I_n(\mathbf{q};c) = J_n(\mathbf{q};c)$.
\end{proof}

\begin{cor}
We have
$$\sum_{F \in F_n} q^{\inv(F)} u^{\lead(F)} c^{\tree(F)} = \sum_{P \in PF_n} q^{\jump(P)} u^{\lucky(P)}c^{\critic(P)}$$
\end{cor}
\begin{proof}
By Theorem~\ref{main}, $I_n(u,q,q^2,\ldots;c) = J_n(u,q,q^2,\ldots;c)$.
Simplifying it by 
\begin{eqnarray*}
\inv(F) &=& \textstyle\sum_i i \cdot \lead_i(F),\\
\lead(F) &=& \lead_0(F),\\
\jump(P) &=& \textstyle\sum_i i \cdot \lucky_i(P) = {n+1 \choose 2} - \abs{P},\mbox{ and}\\
\lucky(P) &=& \lucky_0(P),
\end{eqnarray*}
we are done.
\end{proof}

%The next theorem is the generalization of Theorem~\ref{main}.
%\begin{thm}\label{main2}
%$$I_n(\mathbf{q};c) = J_n(\mathbf{q};c)$$
%where
%\begin{eqnarray*}
%I_n(\mathbf{q};c) &=& \sum_{F \in F_n} q_0^{t_0(F)} q_1^{t_1(F)} q_2^{t_2(F)} \cdots c^{\tree(F)}\\
%J_n(\mathbf{q};c) &=& \sum_{P \in PF_n} q_0^{s_0(F)} q_1^{s_1(F)} q_2^{s_2(F)} \cdots c^{\critic(P)}
%\end{eqnarray*}
%\end{thm}
%
%\begin{proof}
%By the map $\theta_F$, each root of trees in $F$ corresponds to each of right-to-left maximums in $PA(P)^{-1}$. Hence we have $$\tree(F) = \critic(P)$$ where $P=\varphi(F)$. So this is proved with Theorem~\ref{main}.
%\end{proof}

%\begin{cor}\label{main3}
%$$\sum_{F \in F_n} q^{\inv(F)} u^{\lead(F)} c^{\tree(F)} = \sum_{P \in PF_n} q^{\jump(P)} u^{\lucky(P)}
%c^{\critic(P)}$$
%\end{cor}
%\begin{proof}
%By Theorem~\ref{main2}, $I_n(u,q,q^2,\ldots;c) = J_n(u,q,q^2,\ldots;c)$.
%\end{proof}

\begin{cor}
We have
$$\sum_{P \in PF_n} c^{\critic(P)} u^{\lucky(P)} = cu \prod_{i=1}^{n-1}(i + (n-i)u + cu),$$
which have a bijective proof.
\end{cor}
\begin{proof}
By Theorem~\ref{main}, $I_n(u,1,1,\ldots;c) = J_n(u,1,1,\ldots;c)$.
%Substitute $q=1$ in Corollary~\ref{main3}.
We get $$\sum_{F \in F_n} u^{\lead(F)} c^{\tree(F)} = \sum_{P \in PF_n} u^{\lucky(P)} c^{\critic(P)}.$$ Recall the formula in \cite[Eq.(1)]{MR2353128}, $$\sum_{F \in F_n} u^{\lead(F)} c^{\tree(F)} = P_n(1,u,cu)$$ where $P_n(a,b,c)=c \prod_{i=1}^{n-1}(ia + (n-i)b + c).$ Combining above two formulae, we are done.
\end{proof}

Forests and parking functions have not only the same cardinality, but also many equinumerous statistics. The map $\varphi$ corresponds simultaneously between statistics $\inv$, $\lead$, and $\tree$ in forests and statistics $\jump$, $\lucky$, and $\critic$ in parking functions. Also, while the $\varphi$ makes a correspondence between combinatorial objects, the $\theta$ makes a correspondence between vertices in forests and cars in parking functions in detail.

%Now to conclude, forests and parking functions have same structures.

\nocite{MR0335294}

% ----------------------------------------------------------------
\bibliographystyle{alpha}
\bibliography{Shi08a}
\end{document}